\documentclass[11pt]{amsart}
\usepackage{amssymb}
\usepackage{url}
\usepackage{enumerate}
\usepackage{tikz}
\newtheorem{lemma}{\bf Lemma}[section] % putting the percent symbol in front of "[section" prevents the section number from being put in front of each lemma
\newtheorem{theorem}[lemma]{\bf Theorem}
\newtheorem{corollary}[lemma]{\bf Corollary}
\newtheorem{proposition}[lemma]{\bf Proposition}

\newtheorem{notation}[lemma]{\bf Notation}

\newtheorem{observation}[lemma]{\bf Observation}

\DeclareMathOperator{\h}{h}

\begin{document}
\parskip = 0mm
\title[The Endomorphism Ordered Set]{Does the Endomorphism Ordered Set of a Finite Ordered Set Determine the Ordered Set?\\ \small The Cases of Height $1$ and ``Trebled" Ordered Sets}
%\author[ideas of Tim Campion]{ideas of Tim Campion}
%\address{{\tt timcampionmath@gmail.com}}
\author[Jonathan David Farley]{Jonathan David Farley}
\address{Department of Mathematics, Morgan State University, 1700 E. Cold Spring Lane, Baltimore, MD 21251, United States of America, {\tt lattice.theory@gmail.com}}



\begin{abstract}
For ordered sets $X$ and $Y$, let $Y^X$ denote the ordered set of order-preserving maps from $X$ to $Y$, where $f\le g$ in $Y^X$ if $f(x)\le g(x)$ for all $x\in X$.

Let $P$ and $Q$ be finite ordered sets such that $P^P\cong Q^Q$.  It is proven that $P\cong Q$ if $P$ or $Q$ has height at most $1$ or if $P$ and $Q$ are ordered sets of the following form: replace each element of an ordered set with a three-element antichain. The latter is an elaboration of a proof of Tim Campion.
\end{abstract}

\maketitle

%-------------------- new commands ---------------------

\def\Qa{\mathbb{Q}_0}
\def\Qb{\mathbb{Q}_1}
\def\Q{\mathbb{Q}}
\def\card{{\rm card}}
\parskip = 2mm
\parindent = 10mm
\def\Part{{\rm Part}}
\def\P{{\mathcal P}}
\def\Eq{{\rm Eq}}
\def\cld{Cl_\tau(\Delta)}
\def\Csing{{\mathcal C}_{\{*\}}}
\def\Cftwo{{\mathcal C}_{{\rm fin}\rangle1}}
\def\Cinf{{\mathcal C}_{\infty}}
\def\Pcf{{\mathcal P}_{\rm cf}}
\def\Fn{{\mathcal F}_n}
\def\proof{{\it Proof. }}
%\def\qed{\hfill{}$\Box$}

%%%%%%%%%%%%%%%%%%%%% end of commands

\vspace*{-4mm} %----\rangle !! Shift text upwards closer to Abstract !!

%===================== main text ============================
\section{Notation and Motivation}
Notation and definitions not explained in this article can be found in \cite{GraAA}, \cite{DavPriJB}, \cite{SchAF}, \cite{FarBC} and the references therein.  Ordered sets are non-empty. The survey \cite{DufHD} lists many of the important results in the arithmetic of ordered sets.

For an ordered set $P$ and $p\in P$, 
$$
\overset{\circ}{\downarrow} p:=(\downarrow p)-\{p\}
$$
\noindent and 
$$
\underset{\circ}{\uparrow} p:=(\uparrow p)-\{p\}.
$$

We use the notion of {\sl connected} ordered sets and connected {\sl components} of ordered sets and Hashimoto's Refinement Theorem \cite[Section 12.4 and Definitions 2.43 and 2.47]{SchAF}.  The symbol ``$\equiv$" means ``in the same connected component as".  {\sl Directly-irreducible} ordered sets are defined on \cite[p. 436]{FarBC}.

For ordered sets $X$ and $Y$, let $Y^X$ denote the ordered set of order-preserving maps from $X$ to $Y$, ordered as follows: If $f,g\in Y^X$, $f\le g$ if $f(x)\le g(x)$ for all $x\in X$; for $y\in Y$, we denote the constant map by $\langle y\rangle$; $f\lessdot g$ ($f$ is a lower cover of $g$) if there exists $x\in X$ such that $f(x)\lessdot g(x)$ and $f\restriction_{X-\{x\}}=g\restriction_{X-\{x\}}$ \cite[Exercise 1.27]{DavPriJB}.  If $f:P\to P$ is a function and $p,x\in P$, then $f_{p\mapsto x}:P\to P$ is the function that sends $p$ to $x$ but is otherwise the same as $f$.  

For ordered sets $X$, $Y$, and $Z$, $Z^{Y\times X}\cong (Z^Y)^X$, where ``$\cong$" denotes order-isomorphism \cite[Chapter III, Section 2, Theorem 2(4)]{BirFG}.

For $n\in\mathbb N_0$, let $[n]:=\{1,\dots,n\}$; ${\bf n}:=\{0,1,\dots,n-1\}$ is the $n$-element {\sl chain} or totally-ordered set; ${\bf 2}^{\bf 2}$ is order-isomorphic to $\bf 3$; every finite distributive lattice $L$ has the form ${\bf 2}^P$ where $P$ is the dual ${\mathcal J}(L)^\partial$ of the ordered set of join-irreducibles of $L$ \cite[Exercise 5.19(iii)]{DavPriJB}.  The {\sl height} $\h(P)$ of a finite ordered set $P$ is the size of the largest chain in $P$ minus one; for finite ordered sets $X$ and $Y$, $\h(X\times Y)=\h(X)+\h(Y)$ \cite[Chapter I, Section 9, Exercise 4(a)]{BirDH}; $\h(Y^X)=|X|\h(Y)$ \cite[Lemma 15(4)]{FarBC}; in fact, every maximum-sized chain in $Y^X$ has as endpoints a map that is constant on all components; the notation $\mathcal C(Y^X)$ is defined on \cite[p. 436]{FarBC}; we really should use the notation $\mathcal C(X,Y)$.  We will often use \cite[Proposition 1 and Theorem 2]{FarBC}, noting that the order-isomorphisms are canonical.  {\sl Absolute $\mathcal C$-indecomposability} is defined on \cite[p. 436]{FarBC}.  

If $F_3$ is the $3$-element {\sl fence} $\{a,b,c\}$ where $a<b$; $b>c$; and no other strict comparabilities hold, then $F_3^{F_3}$ is the $11$-element ordered set in Figure 1.1.

\begin{center}

    \begin{tikzpicture}[scale=.45]

    \draw[fill] (0,12) circle (.05cm);
    \draw (0,12) -- (-4,8);
    \draw (0,12) -- (-2,8);
    \draw (0,12) -- (4,8);
    \draw (0,12) -- (2,8);
    \draw[fill] (-4,8) circle (.05cm);
    \draw[fill] (-2,8) circle (.05cm);
    \draw[fill] (4,8) circle (.05cm);
    \draw[fill] (2,8) circle (.05cm);

    \draw (-4,8) -- (-4,4);
    \draw (-4,8) -- (2,4);
    \draw (-2,8) -- (-4,4);
    \draw (-2,8) -- (-2,4);
    \draw (2,8) -- (-2,4);
    \draw (2,8) -- (4,4);
    \draw (4,8) -- (2,4);
    \draw (4,8) -- (4,4);
    \draw[fill] (-4,4) circle (.05cm);
    \draw[fill] (-2,4) circle (.05cm);
    \draw[fill] (4,4) circle (.05cm);
    \draw[fill] (2,4) circle (.05cm);

    \draw (-4,4) -- (-4,0);
    \draw (4,4) -- (4,0);
    \draw[fill] (-4,0) circle (.05cm);
    \draw[fill] (4,0) circle (.05cm);

    \draw (-1,-1) node {\bf Figure 1.1. $F_3^{F_3}$};

    \end{tikzpicture}
    
\end{center}
If $C_4$ is the $4$-element {\sl crown} $\{a,b,c,d\}$ where $a,c<b,d$ and no other strict comparabilities hold, then $C_4^{C_4}$ is the ordered set in Figure 1.2 (taken from \cite[p. 54]{DufGH}).  

\begin{center}

    \begin{tikzpicture}[scale=.45]

    \draw[fill] (-4,0) circle (.05cm);
    \draw[fill] (4,0) circle (.05cm);
    \draw[fill] (-7,1) circle (.05cm);
    \draw[fill] (-5,1) circle (.05cm);
    \draw[fill] (-3,1) circle (.05cm);
    \draw[fill] (-1,1) circle (.05cm);
    \draw[fill] (7,1) circle (.05cm);
    \draw[fill] (5,1) circle (.05cm);
    \draw[fill] (3,1) circle (.05cm);
    \draw[fill] (1,1) circle (.05cm);

    \draw[fill] (-7,2) circle (.05cm);
    \draw[fill] (-6,2) circle (.05cm);
    \draw[fill] (-5,2) circle (.05cm);
    \draw[fill] (-3,2) circle (.05cm);
    \draw[fill] (-2,2) circle (.05cm);
    \draw[fill] (-1,2) circle (.05cm);

    \draw[fill] (7,2) circle (.05cm);
    \draw[fill] (6,2) circle (.05cm);
    \draw[fill] (5,2) circle (.05cm);
    \draw[fill] (3,2) circle (.05cm);
    \draw[fill] (2,2) circle (.05cm);
    \draw[fill] (1,2) circle (.05cm);

    \draw[fill] (-7,3) circle (.05cm);
    \draw[fill] (-5,3) circle (.05cm);
    \draw[fill] (-3,3) circle (.05cm);
    \draw[fill] (-1,3) circle (.05cm);

    \draw[fill] (7,3) circle (.05cm);
    \draw[fill] (5,3) circle (.05cm);
    \draw[fill] (3,3) circle (.05cm);
    \draw[fill] (1,3) circle (.05cm);

    \draw[fill] (-4,4) circle (.05cm);
    \draw[fill] (4,4) circle (.05cm);

    \draw[fill] (8,2) circle (.05cm);
    \draw[fill] (9,2) circle (.05cm);
    \draw[fill] (10,2) circle (.05cm);
    \draw[fill] (11,2) circle (.05cm);

    \draw (-4,0) -- (-7,1);
    \draw (-4,0) -- (-5,1);
    \draw (-4,0) -- (-3,1);
    \draw (-4,0) -- (-1,1);

    \draw (4,0) -- (7,1);
    \draw (4,0) -- (5,1);
    \draw (4,0) -- (3,1);
    \draw (4,0) -- (1,1);

    \draw (-7,1) -- (-7,2);
    \draw (-7,1) -- (-2,2);

    \draw (-5,1) -- (-7,2);
    \draw (-5,1) -- (-3,2);

    \draw (-3,1) -- (-3,2);
    \draw (-3,1) -- (1,2);

    \draw (-1,1) -- (-2,2);
    \draw (-1,1) -- (1,2);

    \draw (7,1) -- (7,2);
    \draw (7,1) -- (6,2);

    \draw (5,1) -- (7,2);
    \draw (5,1) -- (3,2);

    \draw (3,1) -- (3,2);
    \draw (3,1) -- (-1,2);

    \draw (1,1) -- (6,2);
    \draw (1,1) -- (-1,2);

    \draw (-7,2) -- (-7,3);
    \draw (-7,2) -- (-5,3);

    \draw (-6,2) -- (-7,3);
    \draw (-6,2) -- (-1,3);

    \draw (-5,2) -- (-5,3);
    \draw (-5,2) -- (-3,3);

    \draw (-1,2) -- (-3,3);
    \draw (-1,2) -- (-1,3);

    \draw (1,2) -- (1,3);
    \draw (1,2) -- (3,3);

    \draw (2,2) -- (1,3);
    \draw (2,2) -- (7,3);

    \draw (5,2) -- (3,3);
    \draw (5,2) -- (5,3);

    \draw (7,2) -- (5,3);
    \draw (7,2) -- (7,3);

    \draw (-7,3) -- (-4,4);
    \draw (-5,3) -- (-4,4);
    \draw (-3,3) -- (-4,4);
    \draw (-1,3) -- (-4,4);

    \draw (7,3) -- (4,4);
    \draw (5,3) -- (4,4);
    \draw (3,3) -- (4,4);
    \draw (1,3) -- (4,4);

    \draw (0,-1) node {\bf Figure 1.2. $C_4^{C_4}$ \cite[p. 54]{DufGH}};

    \end{tikzpicture}
    
\end{center}

In 1984, Duffus posed the following problem:

If $P$ and $Q$ are finite ordered sets and $P^P\cong Q^Q$, must we have $P\cong Q$? \cite[p. 90]{DufHD}  The answer is ``yes" if $P$ is an antichain or, less trivially, if both $P$ and $Q$ are connected \cite[Theorem]{DufWilGI} or, more generally, if $P$ {\it or} $Q$ is connected \cite[Lemma 15(2) and Theorem 16]{FarBC}.

In this article, we show the answer is ``yes" if $P$ or $Q$ has height $1$ (Section 2) or if $P$ and $Q$ are ``trebled" ordered sets (Sections 3 and 4): lexicographic sums where each summand is a three-element antichain.  Informally, replace each element $y$ of an ordered set $Y$ with a three-element set consisting of ``triplets" $y$, $\widetilde y$, and $\widetilde{\widetilde y}$.  Call the trebled ordered set $\widetilde Y$; e.g., the covering graph of $\widetilde{\bf 2}$ is $K_{3,3}$.

Sections 3 and 4 are independent of Section 2.

\section{Height $1$ ordered sets}

In this section, we prove
\begin{theorem} Let $P$ and $Q$ be finite ordered sets where $\h(P)\le1$.  If $P^P\cong Q^Q$, then $P\cong Q$.
\end{theorem}

\begin{lemma} Let $A$ and $B$ be finite ordered sets such that $A$ is connected.  Then $\mathcal C(A^B)$ is connected.
\end{lemma}

\proof Use \cite[Proposition 1(2), Proposition 1(4), and Lemma 9(1)]{FarBC}. \qed

\begin{proposition} Let $P$ and $Q$ be finite, disconnected ordered sets.  Assume that $\phi:P^P\to Q^Q$ is an order-isomorphism. If $Q$ has a connected component of height $\h(Q)$ that is directly irreducible, then $\h(Q)\le\h(P)$ and $|P|\le|Q|$.
\end{proposition}

\proof Say $Q=Q_0+D$, where $D\ne\emptyset$ and $Q_0$ is any component of $Q$ of height $\h(Q)$ that is directly irreducible.

Let $P=P_0+P_1+\cdots+P_m$ where $m\ge1$ and $P_i$ is connected ($0\le i\le m$).

If $P$ is an antichain, then so is $P^P$ and hence $Q^Q$, and hence $Q$, so $|P|=|Q|$ and thus $P\cong Q$.

So assume $\h(P),\h(Q)\ge1$.
For $p\in P$, let $\kappa(p)\in\{0,\dots,m\}$ be such that $p\in P_{\kappa(p)}$.

Pick $q_0\in Q_0$ of maximum height in $Q$.  By \cite[Lemma 15]{FarBC}, there exists $\ell\in\{0,\dots,m\}$ and $p_0,\dots,p_\ell\in P$ such that $p_i$ has height $\h(P)$ for $0\le i\le \ell$, and there exists a partition $\{\Pi_0,\dots,\Pi_\ell\}$ of $\{0,1,\dots,m\}$ into $\ell+1$ parts such that for all $p\in P$ with $\kappa(p)\in\Pi_i$ we have $\phi^{-1}(\langle q_0\rangle)(p)=p_i$ ($0\le i\le\ell$).

We have $\mathcal C(Q_0^Q)=\{g\in Q^Q\mid g\equiv\langle q_0\rangle\}$ and under $\phi^{-1}$ this goes to an ordered set order-isomorphic to $\prod_{i=0}^\ell\prod_{j\in\Pi_i}\mathcal C(P_{\kappa(p_i)}^{P_j})$.

By Hashimoto's Refinement Theorem, for $i\in\{0,1,\dots,\ell\}$ and \cite[Theorem 2(3)]{FarBC}, there exists a direct factor $A_i$ of $P_{\kappa(p_i)}$ such that
$$
\mathcal C(Q_0^{Q_0})\cong\prod_{i=0}^\ell\prod_{j\in\Pi_i}\mathcal C(A_i^{P_j})
$$
\noindent and thus for all $i\in\{0,\dots,\ell\}$, $j\in\Pi_i$ there exists a factor $B_{i,j}$ of $Q_0$ such that $\mathcal C(B_{i,j}^{Q_0})\cong\mathcal C(A_i^{P_j})$.

Fix $i\in\{0,\dots,\ell\}$, $j\in\Pi_i$.  Now, for some $i$, $A_i$ will be non-trivial and connected (because $Q_0$ is non-trivial and $P_{\kappa(p_i)}$ is connected), so we can assume we have such an $i$.

By \cite[Theorem 10]{FarBC}, there exist finite ordered sets $E$, $X$, $Y$, and $Z$ such that
$$
A_i\cong\mathcal C(E^X)\text{,}\ B_{i,j}\cong\mathcal C(E^Y)\text{,}\ Q_0\cong X\times Z\text{,}\ P_j\cong Y\times Z.
$$
Since $A_i$ is non-trivial and connected, then $E$ is non-trivial and connected, and hence contains $\bf2$.

Since $Q_0$ is directly irreducible, then $|X|=1$ or $|Z|=1$.

If $|X|=1$, then $Q_0$ is a direct factor of $P_j$ and hence 
$$
\h(Q)=\h(Q_0)\le\h(P).
$$

If $|Z|=1$, then $A_i$ contains ${\bf 2}^X$.  Since the dual of $X$ can be embedded in ${\bf 2}^X$, then $\h(Q)=\h(Q_0)=\h(X)\le\h(A_i)\le\h(P_{\kappa(p_i)})=\h(P)$.

Since by \cite[Lemma 15(4)]{FarBC}, $|P|\h(P)=\h(P^P)=\h(Q^Q)=|Q|\h(Q)$, we have $|P|\le|Q|$. \qed

\begin{corollary} Let $P$ and $Q$ be finite ordered sets.  Assume $P^P\cong Q^Q$.   If $P$ has a connected component of maximum height that is directly irreducible and so does $Q$, then $\h(Q)=\h(P)$ and $|P|=|Q|$. \qed
\end{corollary}

\begin{observation} Let $P$ be a finite ordered set.  Let $a$ be the number of elements of height $\h(P)$ in $P$ and let $c$ be the number of connected components of $P$.
Then the number of elements of $P^P$ of maximum height is $a^c$.
\end{observation}

\proof See \cite[Lemma 15(3)]{FarBC}. \qed

\begin{lemma} Let $A$ and $B$ be finite connected ordered sets.  Then $\h\big(\mathcal C(B^A)\big)=\h(B^A)$.
\end{lemma}

\proof One proof that $\h(B^A)=|A|\h(B)$ takes a chain in $B$ $$
b_0\lessdot b_1\lessdot\cdots\lessdot b_{\h(B)}
$$
\noindent and scales ${\bf2}^A$ up each covering pair.  These maps are all in $\mathcal C(B^A)$. \qed

\begin{proposition} Let $P$ and $Q$ be finite ordered sets and let 
$$
\varphi:P^P\cong Q^Q.
$$
\noindent Let $k,\ell,m,n$ be such that $1\le k\le m$ and $1\le\ell\le n$; let $P_1,\dots,P_m$ be the $m$ connected components of $P$ such that $P_1,\dots,P_k$ are the components of height $\h(P)$; let $Q_1,\dots,Q_n$ be the $n$ connected components of $Q$ such that $Q_1,\dots,Q_\ell$ are the components of height $\h(Q)$.

For $f\in P^P$ and $i\in\{1,\dots,m\}$, let $\kappa(f,i)\in\{1,\dots,m\}$ be such that $f[P_i]\subseteq P_{\kappa(f,i)}$; similarly, define $\lambda(g,j)$ for $g\in Q^Q$ and $j\in\{1,\dots,n\}$.

Then $\varphi$ restricts to an order-isomorphism from
$$
\mathcal A:=\{f\in(P_1\cup\cdots\cup P_k)^P\mid\text{ for }i=1,\dots,m\text{, }f\restriction_{P_i}\in\mathcal C(P_{\kappa(f,i)}^{P_i})\}
$$ 
\noindent onto
$$
\mathcal B:=\{g\in(Q_1\cup\cdots\cup Q_\ell)^Q\mid\text{ for }j=1,\dots,n\text{, }g\restriction_{Q_j}\in\mathcal C(Q_{\lambda(g,j)}^{Q_j})\}.
$$
The ordered set $\mathcal A$ is canonically order-isomorphic to
$$
\sum_{\psi:\{1,\dots,m\}\to\{1,\dots,k\}}\prod_{i=1}^m\mathcal C(P_{\psi(i)}^{P_i}).
$$

Also, $k^m=\ell^n$.
\end{proposition}

\proof The ordered set $\mathcal A$ equals
$$
\{f\in P^P\mid f\equiv e\in P^P\text{ where }\h(e)=\h(P^P)\}.
$$ 

For the last statement, use \cite[Lemma 9(1) and Theorem 2(3)]{FarBC} to see that each summand is connected. \qed

\begin{corollary} Let all be as in Proposition 2.7.  Assume $P_1$ and $Q_1$ are directly irreducible.  Then $m=n$ and the summand corresponding to the function $\psi:\{1,\dots,m\}\to\{1,\dots,k\}$ such that $\psi(r)=1$ for all $r\in\{1,\dots,m\}$ is such that, via Hashimoto's Strong Refinement Theorem, $\mathcal C(P_1^{P_1})\cong\mathcal C(Q_t^{Q_j})$ for some $j,t\in\mathbb N$ where $1\le t\le\ell$ and $1\le j\le n$ and $P_1\cong Q_j\cong Q_t$.
\end{corollary}

\proof By Corollary 2.4, $\h(P)=\h(Q)$.

Recall \cite[Lemma 9(1) and Theorem 2(3)]{FarBC}.  Connected components in the sum of Proposition 2.7 for $P$ must go to connected components in the corresponding sum for $Q$.  By \cite[Theorem 2(3)]{FarBC}, each factor $\mathcal C(P_1^{P_r})$ where $r\in\{1,\dots,m\}$ is directly irreducible.  That component corresponds to a product $\prod_{j=1}^n\mathcal C(Q_{\chi(j)}^{Q_j})$ for some $\chi:\{1,\dots,n\}\to\{1,\dots,\ell\}$, where each factor is non-trivial since $\h(Q)\ge1$ (because $|P_1|,|Q_1|\ge2$ and $P_1$ and $Q_1$ are connected).

Thus $n\le m$.  By symmetry, $m\le n$, so $m=n$.  Not only that, $Q_{\chi(j)}$ and $\mathcal C(Q_{\chi(j)}^{Q_j})$ are directly indecomposable for all $j\in\{1,\dots,n\}$.  This means that by Hashimoto's Refinement Theorem, $\mathcal C(P_1^{P_1})\cong\mathcal C(Q_t^{Q_j})$ for some $j,t\in\mathbb N$ with $1\le j\le n$ and $1\le t\le \ell$.  By \cite[Theorem 4]{FarBC}, there exist finite, connected ordered sets $E$, $X$, $Y$, and $Z$ such that
$$
P_1\cong\mathcal C(E^X)\text{, }Q_t\cong\mathcal C(E^Y)\text{, }P_1\cong Y\times Z\text{, }Q_j\cong X\times Z.
$$
\noindent If $|Z|=1$, then ${\bf2}^{P_1}$ embeds in $Q_t$ (since $E$ is connected but $|E|\ge2$, since $|P_1|\ge2$).  But then $\h(P)=\h(Q)=\h(Q_t)>\h(P_1)=\h(P)$, a contradiction.

Thus $|Y|=1$, so $Q_j\cong X\times P_1$, and $\h(P_1)\ge\h(Q_j)=\h(X)+\h(P_1)$ and hence $\h(X)=0$, making $X$ an antichain.  As $X$ is connected, $|X|=1$.  Thus $Q_j\cong P_1\cong Q_t$. \qed

\begin{theorem} Let $P$ and $Q$ be finite ordered sets such that $P^P\cong Q^Q$.  Assume that $P$ has a connected component of height $\h(P)$ that is directly irreducible and $Q$ has a connected component of height $\h(Q)$ that is directly irreducible.

Then $P$ and $Q$ have the same height and the same number of connected components of maximum height, the same number of connected components, the same number of elements, and the same number of elements of height $\h(P)$.

There is a bijection between the set of directly-irreducible connected components of $P$ of height $\h(P)$ and the set of those of $Q$ that pairs order-isomorphic components.

Further, every component of $P$ of maximum height that is a product of exactly two directly-irreducible ordered sets is order-isomorphic to a connected component of $Q$, and vice versa.
\end{theorem}

\proof We use the notation of Proposition 2.7.
The first part is Corollary 2.4 and Corollary 2.8 and Observation 2.5: Let $\alpha,\beta\in\mathbb N$ be such that $1\le\alpha\le k$ and $1\le\beta\le \ell$ and $P_1,\dots,P_\alpha,Q_1,\dots,Q_\beta$ are directly irreducible.  By Proposition 2.7 and Corollary 2.8, $k=\ell$ (taking $n$th roots since $m=n$), implying that $P$ and $Q$ have the same number of elements of height $\h(P)$.

Let $\psi:\{1,\dots,m\}\to\{1,\dots,k\}$ be such that $\psi(i)=i$ for $1\le i\le\alpha$ and $\psi(i)=1$ for $\alpha<i\le m$.  Using Proposition 2.7 and \cite[Lemma 9(1)]{FarBC}, this summand of $\mathcal A$ has $m$ directly-irreducible factors, so in the corresponding summand of $\mathcal B$, $\prod_{j=1}^n\mathcal C(Q_{\chi(j)}^{Q_j})$ for some $\chi:\{1,\dots,n\}\to\{1,\dots,\ell\}$, each $Q_{\chi(j)}$ is directly irreducible.

Each $\mathcal C(P_{\psi(i)}^{P_i})$ for $1\le i\le\alpha$ corresponds to a different $j_i\in\{1,\dots,n\}$.  Fix $i\in\{1,\dots,\alpha\}$.  By \cite[Theorem 4]{FarBC}, there exist finite connected ordered sets $E,X,Y,Z$ such that
$$
P_i\cong\mathcal C(E^X)\text{, }Q_{\chi(j_i)}\cong\mathcal C(E^Y)\text{, }P_i\cong Y\times Z\text{, }Q_{j_i}\cong X\times Z.
$$
\noindent Since $P_i$ is non-trivial, so is $E$.  Since $E$ is connected and non-trivial, $\bf2$ embeds in $E$.  If $|Z|=1$, then $\h(Q)=\h(Q_{\chi(j)})\ge\h({\bf2}^{P_i})>\h(P_i)=h(P)$, a contradiction.

Thus $|Y|=1$ so $\h(Q)\ge\h(Q_{j_i})=\h(X)+\h(P_i)=\h(X)+\h(P)=\h(X)+\h(Q)$.  Hence $\h(X)=0$.  Since $X$ is connected, $|X|=1$.

Thus $P_i\cong Q_{\chi(j_i)}\cong Q_{j_i}$.
Use symmetry.

Now assume $\alpha<k$ and $P_{\alpha+1}\cong V\times W$ where $V$ and $W$ are directly irreducible.  Let $\psi:\{1,\dots,m\}\to\{1,\dots,k\}$ be such that
$$
\psi(i)=
\begin{cases}
\alpha+1\ \text{if}\ i=1\\
1\ \text{if}\ i\ne1.
\end{cases}
$$
\noindent Then, as above, the corresponding summand of $\mathcal A$ has $m+1$ directly-irreducible factors.  The summand of $\mathcal B$ it corresponds to must be such that $Q_{\chi(j)}$ is directly irreducible for all $j\in\{1,\dots,\hat{j_0},\dots,n\}$ and $Q_{\chi(j_0)}\cong T\times U$, where $T$ and $U$ are directly-irreducible ordered sets.

If $\mathcal C(V^{P_1})$ corresponds under Hashimoto's Refinement Theorem to $\mathcal C(Q_{\chi(j)}^{Q_j})$ where $j\in\{1,\dots,n\}-\{j_0\}$, then by \cite[Theorem 4]{FarBC} there exist finite connected ordered sets $E,X,Y,Z$ such that
$$
V\cong\mathcal C(E^X)\text{, }Q_{\chi(j)}\cong\mathcal C(E^Y)\text{, }P_1\cong Y\times Z\text{, }Q_{j}\cong X\times Z.
$$
\noindent If $|Z|=1$, then $\h(Q)=\h(Q_{\chi(j)})\ge\h({\bf2}^{P_1})>\h(P_1)=h(P)$, a contradiction.  Thus $|Y|=1$ so
$$
\h(Q)\ge\h(Q_{j})=\h(X)+\h(P_1)=\h(X)+\h(P)=\h(X)+\h(Q).
$$
\noindent  Hence $\h(X)=0$.  Since $X$ is connected, $|X|=1$, so $V\cong Q_{\chi(j)}$.  But then $\h(P)=\h(P_{\alpha+1})>\h(V)=\h(Q_{\chi(j)})=\h(Q)$, a contradiction.

Hence without loss of generality $\mathcal C(V^{P_1})\cong\mathcal C(T^{Q_{j_0}})$.  Again, we have finite connected ordered sets $E,X,Y,Z$ such that
$$
V\cong\mathcal C(E^X)\text{, }T\cong\mathcal C(E^Y)\text{, }P_1\cong Y\times Z\text{, }Q_{j_0}\cong X\times Z.
$$
\noindent If $|Z|=1$, then $\h(Q)=\h(\chi(j_0))>\h(T)\ge\h({\bf2}^{P_1})>\h(P_1)=\h(P)$, a contradiction. Thus $|Y|=1$ and hence $|X|=1$ so $V\cong T$.

Similarly, $W\cong U$.  Thus $P_{\alpha+1}\cong Q_{\chi(j_0)}$. \qed

\begin{lemma} A finite, connected ordered set $P$ of height at most $1$ is either $\bf1$ or directly irreducible.
\end{lemma}

\proof If $\h(P)=0$, then $P$ is a connected antichain, and hence $|P|=1$.

So suppose $\h(P)=1$.  If $P\cong A\times B$ ($A,B$ ordered sets), then $A$ and $B$ are finite and connected but $1=\h(P)=\h(A)+\h(B)$, so $\h(A)=0$ or $\h(B)=0$.  Hence $A$ or $B$ is a connected antichain, so $A\cong\bf1$ or $B\cong\bf1$. \qed

\begin{lemma} Let $P$ and $Q$ be finite ordered sets and let $\phi:P^P\to Q^Q$ be an order-isomorphism.  Then $\h(P)=\h(Q)$ if $\h(P)\in\{0,1\}$.  If $\h(P)=2$, then either $\h(Q)=2$, or else $\h(Q)=3$ and every component $S$ of $Q$ of height $\h(Q)$ has the form ${\bf2}\times R_S$ where $R_S$ is a component of $P$ of height $2$.
\end{lemma}

\proof By \cite[Theorem 16]{FarBC}, we are done if $P$ or $Q$ is connected.  We may proceed more or less as in the first part of Proposition 2.3, only without assuming that $Q_0$ is directly irreducible.  (We must be careful about the choices we make since $Q_0$ is not necessarily directly irreducible.)  As in that proof, if $|X|=1$ or $|Z|=1$, then $\h(Q)\le\h(P)$.

If $\h(A_i)=1$, then by Lemma 2.6, $1=|X|\h(E)$, so $|X|=1$.

So we may assume $\h(A_i)=2$.  As $A_i$ is a factor of the connected $P_{\kappa(p_i)}$, that means $A_i\cong P_{\kappa(p_i)}$.  By Lemma 2.6, we may assume $|X|=2$.  Since $|X|$ is connected (being a factor of the connected $Q_0$), $X\cong\bf2$.

If $Y$ is non-trivial, then $\h(Y)\ge1$, so $\h(Z)\le1$ (since $\h(Y)+\h(Z)\le2$), and hence $\h(Q_0)\le1+1=2$.

So we may assume $|Y|=1$, and hence $Q_0\cong{\bf2}\times P_j$.

We could have started with any component of $Q$ of maximum height.  We may assume $\h(Q)=3$. \qed

{\it Proof of Theorem 2.1.} We are done if $P$ is an antichain, so suppose $\h(P)=1$.  By Lemma 2.11, $\h(Q)=1$.  By Theorem 2.9 and Lemma 2.10, the sets of non-trivial components of $P$ and $Q$ are in bijection with one another, with a bijection pairing order-isomorphic components, and so are the sets of the rest of the components, which are singletons.  Hence $P\cong Q$. \qed

\begin{lemma} A finite, connected, height $1$ ordered set is absolutely $\mathcal C$-indecomposable.
\end{lemma}

\proof Assume $P$ is a finite, connected, height $1$ ordered set and $A$ and $B$ ordered sets such that $P\cong\mathcal C(A^B)$, where $B\ne\emptyset$.  By Lemma 2.10, $P$ is directly irreducible.  By \cite[Lemma 8]{FarBC}, $A$ and $B$ are connected and finite.

By Lemma 2.6 and \cite[Lemma 15(4)]{FarBC}, $1=\h(P)=|B|\h(A)$.  Hence $|B|=1$ and hence $P\cong A$. \qed

\begin{theorem} Let $P$ and $Q$ be finite ordered sets such that 
$$
P^P\cong Q^Q.
$$
\noindent  Assume that $P$ has a connected component of height $\h(P)$ that is absolutely $\mathcal C$-indecomposable and $Q$ has a connected component of height $\h(Q)$ that is directly irreducible.  Then $P\cong Q$.
\end{theorem}

\proof Let $P_1$ be the above component.  By Theorem 2.9, $\h(P)=\h(Q)$, $m=n$, and $k=\ell$ (i.e., $P$ and $Q$ have the same number of components and the same number of components of maximum height). Let $\Psi:[m]\to[k]$ be the map that sends $i\in[m]$ to $1$.  Then by Proposition 2.7, \cite[Lemma 9(1)]{FarBC}, and Hashimoto's Refinement Theorem, for all $i\in[m]$, there exists $j_i\in[n]$ and $\chi(j_i)\in[\ell]$ such that $\mathcal C(P_1^{P_i})\cong\mathcal C(Q_{\chi(j_i)}^{Q_{j_i}})$ and the map $i\mapsto j_i$ ($i\in[m]$) is one-to-one.  

Fix $i\in[m]$.  By \cite[Theorem 4]{FarBC}, there exist connected, ordered sets $E,X,Y,Z$ such that
$$
P_1\cong\mathcal C(E^X)\text{, }Q_{\chi(j_i)}\cong\mathcal C(E^Y)\text{, }P_i\cong Y\times Z\text{, }Q_{j_i}\cong X\times Z.
$$
\noindent Hence $P_1\cong E$ and $|X|=1$. But then $\h(Q_{\chi(j_i)})=|Y|\h(P_1)$, which would violate $\h(P)=\h(Q)$ unless $|Y|=1$ and hence $P_1\cong Q_{\chi(j_i)}$ and $P_i\cong Q_{j_i}$.

Thus $P\cong Q$. \qed

We get another proof of Theorem 2.1: We just need Lemmas 2.10 and 2.11, showing that $\h(Q)=1$.

The next step would be height $2$ or an extension of Theorem 2.9 that does not assume a component is directly irreducible.

\section{Trebled Ordered Sets}

Let $(R,\le_R)$ be a finite ordered set.  Define an equivalence relation $\sim$ on $R$ as follows: for $r,r'\in R$, $r\sim r'$ if $\underset{\circ}{\uparrow} r=\underset{\circ}{\uparrow} r'$ and $\overset{\circ}{\downarrow} r=\overset{\circ}{\downarrow} r'$.  Assume that the cardinality of every equivalence class is a multiple of $3$. 

Partition each equivalence class into three-element subsets and let $T$ be the collection of these three-element subsets. Define a binary relation $<_T$ on $T$ as follows: for $t,u\in T$, $t<_T u$ if $r<_R s$ for some $r\in t$, $s\in u$ (equivalently, if $r<_R s$ for all $r\in t$ and $s\in u$).  Then $<_T$ is irreflexive and transitive, so $(T,\le_T)$ is an ordered set.

If $(T',\le_T')$ is another such ordered set similarly defined from $R$, define an order-isomorphism from $T$ to $T'$ by arbitrarily mapping the three-element subsets of the equivalence class $[r]$ ($r\in R$) that are elements of $T$ bijectively to the three-element subsets of the equivalence class $[r]$ that are elements of $T'$.

As a corollary, if $A$ and $B$ are finite ordered sets and $\widetilde A\cong\widetilde B$, then $A\cong B$.

\begin{lemma} Let $A$ be an ordered set.  Then any connected component of $\widetilde A$ is either a singleton or of the form $\widetilde B$ for some ordered set $B$.
\end{lemma}

\proof Let $C$ be such a component.  If $|C|>1$, then for all $c\in C$, there exists $d\in C$ such that $c<d$ or $c>d$. Hence $\widetilde c<d$ or $\widetilde c>d$, respectively, so $\widetilde c,\widetilde{\widetilde c}\in C$.  Select exactly one element from the three-element summand of the lexicographic sum containing $c$ for each $c\in C$ and let $B$ be that ordered set. \qed

\begin{observation} Let $A$ be an ordered set such that $\widetilde A$ is connected and $A$ has a minimal element (equivalently, $\widetilde A$ has a minimal element).  Then $\widetilde A$ is directly irreducible.
\end{observation}

\proof Since $A\ne\emptyset$, $|\widetilde A|>1$.  Let $B$ and $C$ be ordered sets such that $\phi:\widetilde A\cong B\times C$.  Let $m$ be a minimal element of $\widetilde A$ and let $\phi(m)=(b,c)$.  Since $\widetilde A$ is connected and $\{m,\widetilde m\}$ is an antichain, there exists $n\in\widetilde A$ such that $m<n$ (and so $\widetilde m<n$).  Let $\phi(n)=(b',c')$ and, without loss of generality, $b=b'$ and $c<c'$.  Let $\phi(\tilde m)=(b'',c'')$.  Then $b''\le b$ and $c''\le c'$.  We cannot have $c''\le c$ since $\{m,\widetilde m\}$ is an antichain.

If $|B|>1$, then, because $B$ is connected and $m$ minimal, there exists $\overline b\in B$ such that $b<\overline b$.  Then $\phi(m)<(\overline b,c)$, so $\phi(\widetilde m)=(b'',c'')<(\overline b,c)$.  Thus $c''\le c$, a contradiction. \qed

\section{Endomorphism Ordered sets of Finite Trebled Ordered Sets Determine the Ordered et}

In this section, we prove

\begin{theorem} If $P$ and $Q$ are finite ordered sets and ${\widetilde P}^{\widetilde P}\cong{\widetilde Q}^{\widetilde Q}$, then $P\cong Q$.
\end{theorem}

The results come from a conjectural argument of Tim Campion \cite{CamBD}, who gave the first author permission to publish it \cite{CamBE}.

\begin{notation} \rm For the rest of this section, let $P$ and $\overline P$ be finite ordered sets.  (Notions that we define for $P$, when applied to $\overline P$, will have an overbar.) Let $n:=|P|$ and $h:=\h(P)$.  Let $D=\{d_0,d_1,\dots,d_{nh}\}$ be a maximum-sized chain in $P^P$ where $d_0>d_1>\dots>d_{nh}$.  For 
$$
i\in\{0,1,\dots,nh-1\}\text{,}
$$
\noindent let $p_i\in P$ be such that $d_i(p_i)\gtrdot d_{i+1}(p_i)$.  For $p\in P$, let
$$
C_p:=\{d_i(p)\mid 0\le i\le nh\}\text{,}
$$
\noindent a chain of height $h$ in $P$ (from the way one proves \cite[Lemma 15(4)]{FarBC}) with top element $1_{C_p}$ and bottom element $0_{C_p}$.  For $i\in\{0,1,\dots,nh\}$, let
$$
D_i:=\{f\in P^P\mid f\ge d_i\text{ and } f(p)\in C_p\text{ for all }p\in P\}\text{;}
$$
\noindent for $i\in\{0,1,\dots,nh-1\}$, define $a_i:P\to P$ inductively for all $p\in P$ (we show it is well-defined later) by
$$
a_i(p)=\begin{cases}
d_i(p_i)\ \text{if}\ p=p_i\text{;}\\
\max\{x\in C_p\mid x\le d_{i+1}(p_i)\text{ and }x\le a_i(q)\\
\qquad\qquad\text{ for all }q\in\downarrow p_i\text{ such that }p\lessdot q\}\ \text{if}\  p<p_i\text{;}\\
1_{C_p}\ \text{if}\ p\nleq p_i\text{;}
\end{cases}
$$
\noindent and define $b_i:=(a_i)_{p_i\mapsto d_{i+1}(p_i)}$.
\end{notation}

\begin{observation} For $p,q\in P$ such that $p<q$, we have $0_{C_p}=0_{C_q}$ and $1_{C_p}=1_{C_q}$. 
\end{observation}

\proof The maximal and minimal elements of a chain of maximum height in $P^P$ are constant on each connected component of $P$. \qed

\begin{lemma} Let $i\in\{0,1,\dots,nh\}$.  Then the join and meet of every pair of elements in $D_i$ exist in $P^P$.  Indeed, $D_i$ is a bounded distributive lattice with least element $d_i$.  For $e,f\in D_i$, $p\in P$, $(e\vee f)(p)=\max\{e(p),f(p)\}$ and $(e\wedge f)(p)=\min\{e(p),f(p)\}$.
\end{lemma}

\proof Clearly $d_i$ is the least element of $D_i$.  Let $e,f\in D_i$.  Define $g,h:P\to P$ for all $p\in P$ by $g(p)=\max\{e(p),f(p)\}$ and $h(p)=\min\{e(p),f(p)\}$.  For all $p\in P$, $g(p),h(p)\in C_p$ and $g(p),h(p)\ge d_i(p)$.

Assume $p,q\in P$ are such that $p\le q$.  We will show $g(p)\le g(q)$ and $h(p)\le h(q)$.  Without loss of generality, $g(p)=e(p)$ and $g(q)=f(q)$, so $h(p)=f(p)$ and $h(q)=e(q)$.  Thus $g(p)=e(p)\le e(q)\le f(q)=g(q)$ and $h(q)=e(q)\ge e(p)\ge f(p)=h(p)$.  Hence $g,h\in D_i$.  Thus $g=e\vee f$ and $h=e\wedge f$ in $P^P$.

Because $\vee$ and $\wedge$ are calculated pointwise and each $C_p$ is a chain ($p\in P$), $D_i$ is a distributive lattice. \qed

\begin{corollary} There is an order-embedding from $D_{nh}$ into ${\bf (h+1)}^P$.  It is an order-isomorphism if $|\{C_p\mid p\in P\}|=1$.
\end{corollary}

\proof Each $C_p$ is order-isomorphic to $\bf h+1$ via the height function, which is order-preserving. \qed

\begin{proposition} Let $i\in\{0,1,\dots,nh-1\}$.  The function $a_i$ is well-defined and $a_i\in D_i$.  It has a unique lower cover $b$ in $P^P$ such that $b\ngeq d_i$ but $b\ge d_{i+1}$, namely $b_i$.

Moreover, $b_i\wedge e$ exists in $P^P$ for all $e\in D_i$ and
$$
D_{i+1}=D_i\cup\{b_i\wedge e\mid e\in D_i\}.
$$

If $|\{C_p\mid p\in P\}|=1$, then $a_i(p)=d_{i+1}(p_i)$ for all $p\in\overset{\circ}{\downarrow} p_i$.
\end{proposition}

\proof The function $a_i$ is well-defined, since by Observation 4.3 the maxima exist and $1_{C_p}=1_{C_q}$ for all $p,q\in P$ such that $p<q$.  By construction, $a_i\in P^P$.

For $q\in\overset{\circ}{\downarrow} p_i$, $d_i(q)=d_{i+1}(q)\le d_{i+1}(p_i)$; if our induction hypothesis is that $d_i(q)\le a_i(q)$ for $q\in\downarrow p_i\cap\underset{\circ}{\uparrow} p$ ($p\in\overset{\circ}{\downarrow} p_i$), then $d_i(p)\le d_i(q)\le a_i(q)$; $d_i(p)\in C_p$; so the maximum $a_i(p)\ge d_i(p)$.  Hence $d_i\le a_i$, and $a_i\in D_i$.

We have $b_i\in P^P$; $b_i\in D_{i+1}$; and $b_i\lessdot a_i$.  Also, $d_i\nleq b_i$ since 
$$
d_i(p_i)\nleq d_{i+1}(p_i)=b_i(p_i).
$$

Now assume $b\in P^P$; $b\lessdot a_i$; $d_{i+1}\le b$; and $d_i\nleq b$.  Then there exists $p\in P$ such that $d_i(p)\nleq b(p)$, but $d_{i+1}(p)\le b(p)$. Hence $d_i(p)\ne d_{i+1}(p)$, so $p=p_i$ and $b=(a_i)_{p_i\mapsto b(p_i)}$.  Now $d_{i+1}(p_i)\le b(p_i)<a_i(p_i)=d_i(p_i)$ and $d_{i+1}(p_i)\lessdot d_i(p_i)$, so $b(p_i)=d_{i+1}(p_i)$ and thus $b=b_i$.

For all $e\in D_i$, we have $b_i,e\in D_{i+1}$, so $b_i\wedge e$ exists in $P^P$ by Lemma 4.4 and $b_i\wedge e\in D_{i+1}$.

Let $f\in D_{i+1}-D_i$.  Now $d_{i+1}\le f$ but $d_i\nleq f$, so there exists $p\in P$ such that $d_i(p)\nleq f(p)$ but $d_{i+1}(p)\le f(p)$, so $p=p_i$.  Since $d_{i+1}(p_i)\le f(p_i)\in C_{p_i}$ but $d_i(p_i)\nleq f(p_i)$, then $f(p_i)=d_{i+1}(p_i)$.

Let $e:=d_i\vee f\in D_i$.  We now show $b_i\wedge e=f$: Let $p\in P$.  If $p<p_i$, then $d_i(p)\le a_i(p)$ implies $d_i(p)\le b_i(p)$; also $f(p)\le f(p_i)=d_{i+1}(p_i)\le a_i(p_i)$; so $f(p)\le f(q)\le a_i(q)$ for all $q\in\downarrow p_i\cap \underset{\circ}{\uparrow} p$ by induction, and hence 
$$
f(p)\le a_i(p)=b_i(p)\text{;}
$$
\noindent thus $[b_i\wedge(d_i\vee f)](p)=d_i(p)\vee f(p)=d_{i+1}(p)\vee f(p)=f(p)$.

Further, 
$$
b_i(p_i)\wedge[d_i(p_i)\vee f(p_i)]=d_{i+1}(p_i)\wedge[d_i(p_i)\vee d_{i+1}(p_i)]=d_{i+1}(p_i)=f(p_i).
$$
\noindent If $p\nleq p_i$, then 
$$
b_i(p)\wedge[d_i(p)\vee f(p)]=1_{C_p}\wedge[d_i(p)\vee f(p)]=d_i(p)\vee f(p)=d_{i+1}(p)\vee f(p)=f(p).
$$

Hence in all cases $(b_i\wedge e)(p)=f(p)$.

If $|\{C_p\mid p\in P\}|=1$, then $a_i(p)=d_{i+1}(p_i)$ meets the definition for $p\in\overset{\circ}{\downarrow} p_i$. \qed

\begin{lemma} Let $i\in\{0,1,\dots,nh-1\}$.  Let $\alpha\in D_i-\downarrow a_i$ have a unique lower cover $\beta$ in $P^P$ such that $d_{i+1}\le\beta$ and $d_i\nleq\beta$.

Then $\alpha(p_i)>d_i(p_i)$; and $\alpha(p_i)$ has a lower cover $\beta(p_i)$ in $P$ such that $d_{i+1}(p_i)\le\beta(p_i)$ and $d_i(p_i)\nleq\beta(p_i)$.  

Now assume $|\{C_p\mid p\in P\}|=1$.  For $p\in\overset{\circ}{\downarrow} p_i$, $\alpha(p)\le d_{i+1}(p_i)$.  If $\alpha$ is a maximal element of
$$
\{a\in D_i-\downarrow a_i\mid a\text{ has a unique lower cover }b\text{ in }P^P
$$
$$
\text{ such that }d_{i+1}\le b\text{ and }d_i\nleq b\}\text{,}
$$
\noindent then $\alpha>a_i$ and $\alpha(p)=d_{i+1}(p_i)$ for all $p\in\overset{\circ}{\downarrow} p_i$.
\end{lemma}

\proof Let $q\in P$ be such that $d_i(q)\nleq\beta(q)$.  Since $d_i(q)\le\alpha(q)$, we have $\beta(q)\lessdot\alpha(q)$.  Since $d_{i+1}(q)\le\beta(q)$, we have $q=p_i$.

Let $p\in P$ be maximal such that $\alpha(p)\nleq a_i(p)$.  Since $\alpha(p)\in C_p$, we know $p\le p_i$.  Assume for a contradiction that $p<p_i$.

We have $\alpha(p_i)\ge d_i(p_i)$ and, by maximality of $p$, 
$$
\alpha(p_i)\le a_i(p_i)=d_i(p_i).
$$
\noindent  Hence $\alpha(p_i)=d_i(p_i)$.

Now $d_i(p_i)=\alpha(p_i)\gtrdot\beta(p_i)\ge d_{i+1}(p_i)$, so $\beta(p_i)=d_{i+1}(p_i)$.  Thus $\alpha(p)=\beta(p)\le\beta(p_i)=d_{i+1}(p_i)$ and $\alpha(p)\in C_p$, so by the inductive definition of $a_i$, $\alpha(p)\le a_i(p)$, a contradiction.

Hence $\alpha(p_i)\nleq a_i(p_i)=d_i(p_i)$.

Now assume $C=C_p$ for all $p\in P$.  Define $a:P\to P$ for all $p\in P$ by
$$
a(p)=\begin{cases}
1_C\ \text{if}\ p\nleq p_i\text{;}\\
\alpha(p_i)\ \text{if}\ p=p_i\text{;}\\
d_{i+1}(p_i)\ \text{if}\ p<p_i\text{.}
\end{cases}
$$
\noindent Then $a\in P^P$.

Let $p\in\overset{\circ}{\downarrow} p_i$.  Then $\alpha(p)\in C_p$ and $\alpha(p)\le\beta(p_i)$. Since $d_i(p_i)\nleq\beta(p_i)$, then $d_i(p_i)\nleq\alpha(p)$, so $\alpha(p)<d_i(p_i)$ in $C$, and hence $\alpha(p)\le d_{i+1}(p_i)$.

Thus $\alpha\le a$.  Obviously $a\nleq a_i$.  Define $b:=a_{p_i\mapsto\beta(p_i)}$.  Then $b\in P^P$ and $b\lessdot a$.  We have $d_{i+1}\le b$ and $d_i\nleq b$.

Now assume $\widehat b\in P^P$ is such that $\widehat b\lessdot a$; $d_{i+1}\le\widehat b$; and $d_i\nleq\widehat b$.  As in the first paragraph of this proof, $\widehat b(p_i)\lessdot a(p_i)=\alpha(p_i)$; $d_i(p_i)\nleq\widehat b(p_i)$; and $d_{i+1}(p_i)\le\widehat b(p_i)$.

Thus $\alpha_{p_i\mapsto\widehat b(p_i)}\in P^P$; $\alpha_{p_i\mapsto\widehat b(p_i)}\lessdot\alpha$; $d_{i+1}\le\alpha_{p_i\mapsto\widehat b(p_i)}$; and 
$$
d_i\nleq \alpha_{p_i\mapsto\widehat b(p_i)}.
$$
\noindent By uniqueness, $\alpha_{p_i\mapsto\widehat b(p_i)}=\beta$ so $\widehat b(p_i)=\beta(p_i)$ and hence $\widehat b=b$.

Thus $a$ has the uniqueness property.

Now assume $\alpha$ has the maximality property in the statement of the lemma. By maximality, $\alpha=a$.

We see that $\alpha>a_i$. \qed

\begin{proposition} Assume that $P$ and $\overline P$ are trebled ordered sets.  Let 
$$
\phi:P^P\to\overline P^{\overline P}
$$
\noindent be an order-isomorphism.

For $j\in\{0,1,\dots,nh\}$, let $\overline d_j:=\phi(d_j)$.  (Note that $nh=\overline n\overline h$.)  Let $\overline D:=\{\overline d_0,\overline d_1,\dots,\overline d_{nh}\}$.  Fix $i\in\{0,1,\dots,nh-1\}$ and assume $\phi[D_i]=\overline D_i$.

Assume that $|\{C_p\mid p\in P\}|=1$.

Then $a_i$ is the largest element $\alpha\in D_i$ that has a unique lower cover $\beta\in P^P$ such that $d_{i+1}\le\beta$ and $d_i\nleq\beta$.

Furthermore, $\overline a_i$ is the largest element $\overline\alpha\in\overline D_i$ that has a unique lower cover $\overline\beta\in{\overline P}^{\overline P}$ such that $\overline d_{i+1}\le\overline\beta$ and $\overline d_i\nleq\overline\beta$.

We have $\phi(a_i)=\overline a_i$ and $\phi[D_{i+1}]=\overline D_{i+1}$.
\end{proposition}

\proof Assume for a contradiction that there exists $\alpha$ maximal in $D_i$ with respect to the properties that (1) $\alpha\nleq a_i$ and (2) $\alpha$ has a unique lower cover $\beta\in P^P$ such that $d_{i+1}\le\beta$ and $d_i\nleq\beta$.  By Lemma 4.7, $\alpha>a_i$ and $\alpha(p)=d_{i+1}(p_i)$ for all $p\in\overset{\circ}{\downarrow} p_i$.

Note that $\beta(p_i)>d_{i+1}(p_i)$ since $\beta(p_i)\ge d_{i+1}(p_i)$ and 
$$
d_{i+1}(p_i)<d_i(p_i)=a_i(p_i)<\alpha(p_i)
$$
\noindent (see Lemma 4.7) but $\beta(p_i)\lessdot\alpha(p_i)$.  Without loss of generality, $d_i(p_i)\nleq\widetilde{\beta(p_i)}$, because we have triplets.  (See Lemma 4.7.)

Then $\beta_{p_i\mapsto\widetilde{\beta(p_i)}}\in P^P$ (since $\alpha(p)=\beta(p)$ for $p\in P-\{p_i\}$); $\beta_{p_i\mapsto\widetilde{\beta(p_i)}}\lessdot\alpha$; $d_{i+1}\le \beta_{p_i\mapsto\widetilde{\beta(p_i)}}$; and $d_i\nleq\beta_{p_i\mapsto\widetilde{\beta(p_i)}}$.  By uniqueness, $\beta=\beta_{p_i\mapsto\widetilde{\beta(p_i)}}$, so $\beta(p_i)= \widetilde{\beta(p_i)}$, a contradiction.

Now let $\overline\alpha:=\phi(a_i)$.

Assume for a contradiction that $\overline\alpha\ne\overline a_i$.

By what we have shown, $\overline\alpha>\overline a_i$.

Let $\overline\beta$ be the lower cover of $\overline\alpha$ in ${\overline P}^{\overline P}$ such that $\overline d_{i+1}\le\overline\beta$ but $\overline d_i\nleq\overline\beta$.  Let $\overline p_i$ be the element $\overline p\in\overline P$ such that $\overline d_{i+1}(\overline p)\ne\overline d_i(\overline p)$.  We know $\overline\beta(\overline p_i)\lessdot\overline\alpha(\overline p_i)$.

\quad {\it Case 1. $\overline d_{i+1}(\overline p_i)\ne 0_{\overline C_{\overline p_i}}$}

For all $x\in D_i$, $b_i\wedge x$ exists in $P^P$ by Proposition 4.6.  Hence, for all $\overline x\in\overline D_i$, $\overline\beta\wedge\overline x$ exists in ${\overline P}^{\overline P}$.  Define $\overline f:\overline P\to\overline P$ for all $\overline p\in\overline P$ as follows:
$$
\overline f(\overline p):=\begin{cases}
\widetilde{\overline d_{i+1}(\overline p_i)}\ \text{if}\ \overline p\ge \overline p_i\text{;}\\
0_{\overline C_{\overline p}}\ \text{if}\ \overline p\ngeq \overline p_i\text{.}
\end{cases}
$$
\noindent We show that $\overline f\in{\overline P}^{\overline P}$: First use Observation 4.3.  Next, let $\overline q,\overline r\in\overline P$ be such that $\overline q<\overline r$.  If $\overline q\ngeq\overline p_i$ and $\overline r\ge\overline p_i$, then, as $\overline d_{i+1}(\overline p_i)\in\overline C_{\overline p_i}$, but $\overline d_{i+1}(\overline p_i)\ne0_{\overline C_{\overline p_i}}$, we have $0_{\overline C_{\overline q}}=0_{\overline C_{\overline r}}=0_{\overline C_{\overline p_i}}<\overline d_{i+1}(\overline p_i)$, so $\overline f(\overline q)\le\overline f(\overline r)$.

As $\overline\beta(\overline p_i)>\overline d_{i+1}(\overline p_i)$ (because by Lemma 4.7, 
$$
\overline\alpha(\overline p_i)>\overline d_i(\overline p_i)\text{;}
$$
\noindent $\overline d_i(\overline p_i)>\overline d_{i+1}(\overline p_i)$; $\overline\beta(\overline p_i)\ge\overline d_{i+1}(\overline p_i)$; and $\overline\beta(\overline p_i)\lessdot\overline\alpha(\overline p_i)$), we have $\overline\beta(\overline p_i)>\widetilde{\overline d_{i+1}(\overline p_i)}$.  Let $\overline p\in\overline P$.  If $\overline p\ge\overline p_i$, then $\overline\beta(\overline p)\ge\overline\beta(\overline p_i)>\widetilde{\overline d_{i+1}(\overline p_i)}=\overline f(\overline p)$.  If $\overline p\ngeq\overline p_i$, then $\overline p\ne\overline p_i$, so $\overline\beta(\overline p)=\overline\alpha(\overline p)\in\overline C_{\overline p}$ and $\overline\beta(\overline p)\ge\overline f(\overline p)$.  Hence $\overline\beta\ge\overline f$.

Let $\overline p\in\overline P$.  If $\overline p\ge\overline p_i$, then $\overline d_i(\overline p)\ge\overline d_i(\overline p_i)>\overline d_{i+1}(\overline p_i)$, so $\overline d_i(\overline p)>\overline f(\overline p)$.  If $\overline p\ngeq \overline p_i$, then $\overline d_i(\overline p)\in\overline C_{\overline p}$, so $\overline d_i(\overline p)\ge\overline f(\overline p)$.  Hence $\overline d_i\ge\overline f$.

Thus $\overline d_{i+1}=\overline\beta\wedge\overline d_i\ge\overline f$, so $\overline d_{i+1}(\overline p_i)\ge\overline f(\overline p_i)=\widetilde{\overline d_{i+1}(\overline p_i)}$, a contradiction.

\quad {\it Case 2. $\overline \alpha(\overline p_i)\ne 1_{\overline C_{\overline p_i}}$}

For all $x\in D_i$, $b_i\vee x$ exists in $P^P$.  Hence, for all $\overline x\in\overline D_i$, $\overline\beta\vee\overline x$ exists in ${\overline P}^{\overline P}$.  Indeed, $\overline\beta\vee\overline d_i=\overline\alpha$.

Note that $\overline\alpha(\overline p_i)>\overline\beta(\overline p_i),\overline d_i(\overline p_i)$ and for all $p\in\overset{\circ}{\downarrow} p_i$, $\overline\alpha(\overline p)\le\overline\beta(\overline p_i)$.

Define $\overline g:\overline P\to \overline P$ for all $\overline p\in\overline P$ by
$$
\overline g(\overline p):=\begin{cases}
\widetilde{\overline \alpha(\overline p_i)}\ \text{if}\ \overline p\le \overline p_i\text{;}\\
1_{\overline C_{\overline p}}\ \text{if}\ \overline p\nleq \overline p_i\text{.}
\end{cases}
$$
\noindent We show that $\overline g\in{\overline P}^{\overline P}$.  First, use Observation 4.3.  Next, let $\overline q,\overline r\in\overline P$ be such that $\overline q<\overline r$; $\overline q\le\overline p_i$; and $\overline r\nleq\overline p_i$.  Since $\overline\alpha\in\overline D_i$, we have $\alpha(\overline p_i)\in\overline C_{\overline p_i}$.  Since $\overline\alpha(\overline p_i)\ne 1_{\overline C_{\overline p_i}}$, we have $\overline g(\overline q)=\widetilde{\overline\alpha(\overline p_i)}<1_{\overline C_{\overline p_i}}=1_{\overline C_{\overline q}}=1_{\overline C_{\overline r}}=\overline g(\overline r)$ by Observation 4.3.

Let $\overline p\in\overline P$.  If $\overline p\le\overline p_i$, then $\overline\beta(\overline p)\le\overline\beta(\overline p_i)<\widetilde{\overline\alpha(\overline p_i)}=\overline g(\overline p)$ and 
$$
\overline d_i(\overline p)\le\overline d_i(\overline p_i)<\widetilde{\overline\alpha(\overline p_i)}=\overline g(\overline p)
$$
\noindent (using Lemma 4.7).  If $\overline p\nleq\overline p_i$, then $\overline d_i(\overline p)\le\overline\alpha(\overline p)$; but $1_{\overline C_{\overline p}}=\overline a_i(\overline p)\le\overline\alpha(\overline p)\in\overline C_{\overline p}$, so $\overline\beta(\overline p)=\overline\alpha(\overline p)\le\overline g(\overline p)$ and $\overline d_i(\overline p)\le\overline g(\overline p)$.  Hence $\overline\beta,\overline d_i\le\overline g$.  Therefore $\overline\alpha\le\overline g$, so $\overline\alpha(\overline p_i)\le\overline g(\overline p_i)=\widetilde{\overline\alpha(\overline p_i)}$, a contradiction.

\quad {\it Case 3. It is not the case that $\overline d_i(\overline p_i) \lessdot\overline\alpha(\overline p_i)$}

Since $\overline d_i(\overline p_i),\overline\alpha(\overline p_i)\in\overline C_{\overline p_i}$, there exists $\overline t\in\overline C_{\overline p_i}$ such that 
$$
\overline d_i(\overline p_i)<\overline t<\overline\alpha(\overline p_i).
$$

We show that $(\overline a_i)_{\overline p_i\mapsto\overline t}\in{\overline P}^{\overline P}$.  Let $\overline q,\overline r\in\overline P$ be such that $\overline q<\overline r$.  If $\overline q=\overline p_i$, then $(\overline a_i)_{\overline p_i\mapsto\overline t}(\overline r)=\overline a_i(\overline r)=1_{\overline C_{\overline r}}=1_{\overline C_{\overline p_i}}>\overline t=(\overline a_i)_{\overline p_i\mapsto\overline t}(\overline q)$.  If $\overline r=\overline p_i$, then $(\overline a_i)_{\overline p_i\mapsto\overline t}(\overline q)=\overline a_i(\overline q)\le\overline d_{i+1}(\overline p_i)<\overline t=(\overline a_i)_{\overline p_i\mapsto\overline t}(\overline r)$.

Note that $(\overline a_i)_{\overline p_i\mapsto\overline t}\in\overline D_i$ and $(\overline a_i)_{\overline p_i\mapsto\overline t}\ge\overline d_i,\overline a_i$.

\quad {\it Claim 1. $\overline\beta\wedge\overline a_i=\overline b_i$.}

{\it Proof of claim.}  We have $\overline\alpha>\overline a_i$ and $\overline\beta(\overline p_i)\ge\overline d_{i+1}(\overline p_i)=\overline b_i(\overline p_i)$, so $\overline\beta\ge\overline b_i$, but $\overline\beta\ngeq\overline d_i$ and $\overline a_i\ge\overline d_i$, so $\overline\beta\ngeq\overline a_i$. Use Proposition 4.6 for the existence of the meet.\qed

\quad {\it Claim 2. There exists $\overline u\in\overline P$ such that $\overline d_{i+1}(\overline p_i)<\overline u<\overline\beta(\overline p_i)$ and $\overline u\lessdot\overline t$ for which $(\overline a_i)_{\overline p_i\mapsto\overline u}$ is in ${\overline P}^{\overline P}$ and equals $\overline\beta\wedge(\overline a_i)_{\overline p_i\mapsto\overline t}$ in ${\overline P}^{\overline P}$.}

{\it Proof of claim.}  In ${\overline P}^{\overline P}$, $\overline\beta\wedge(\overline a_i)_{\overline p_i\mapsto\overline t}$ exists by Proposition 4.6, and
$$
\overline b_i=\overline\beta\wedge \overline a_i\le\overline\beta\wedge(\overline a_i)_{\overline p_i\mapsto\overline t}\le\overline\beta.
$$
\noindent  For $\overline p\in\overline P-\{\overline p_i\}$, $\overline a_i(\overline p)=\overline b_i(\overline p)\le[\overline\beta\wedge(\overline a_i)_{\overline p_i\mapsto\overline t}](\overline p)\le(\overline a_i)_{\overline p_i\mapsto\overline t}(\overline p)=\overline a_i(\overline p)$.  Let $\overline u:=[\overline\beta\wedge(\overline a_i)_{\overline p_i\mapsto\overline t}](\overline p_i)$.

Note that $\overline d_{i+1}(\overline p_i)\le\overline u\le\overline\beta(\overline p_i),\overline t$.  Since $\overline d_i(\overline p_i)\le\overline t$ but $\overline d_i(\overline p_i)\nleq\overline\beta(\overline p_i)$, we have $\overline u<\overline t$.  Since $\overline\beta(\overline p_i)\lessdot\overline\alpha(\overline p_i)$ but $\overline u<\overline t<\overline\alpha(\overline p_i)$, then $\overline u<\overline\beta(\overline p_i)$.

Assume for a contradiction that $\overline u=\overline d_{i+1}(\overline p_i)$.  Then $(\overline a_i)_{\overline p_i\mapsto\overline u}=\overline b_i$ so $\overline\beta\wedge(\overline a_i)_{\overline p_i\mapsto\overline t}=\overline\beta\wedge\overline a_i$ (by Claim 1) but $\overline a_i<(\overline a_i)_{\overline p_i\mapsto\overline t}$.  Also,
$$
\overline\beta\vee(\overline a_i)_{\overline p_i\mapsto\overline t}=\overline\beta\vee\overline a_i=\overline\alpha.
$$  
\noindent(The joins exist by Lemma 4.4 applied to $D_{i+1}$.) But $\{\overline\beta,\overline\alpha,(\overline a_i)_{\overline p_i\mapsto\overline t},\overline b_i,\overline a_i\}$ is a subset of a distributive lattice order-isomorphic to $D_{i+1}$ (use Claim 1 and Lemma 4.4 for $D_{i+1}$ to see $\overline b_i\in\phi[D_{i+1}]$) even though $\overline a_i<(\overline a_i)_{\overline p_i\mapsto\overline t}$.  (See Figure 4.1.) This is a contradiction \cite[Exercise 4.9]{DavPriJB}. 

\begin{center}

    \begin{tikzpicture}[scale=.45]

    \draw[fill] (0,12) circle (.05cm);
    \draw (0,12) -- (-6,6);
    \draw (0,12) -- (4,8);
    \draw (4,8) -- (4,4);
    \draw (-6,6) -- (0,0);
    \draw (4,4) -- (0,0);

    \draw[fill] (0,12) circle (.05cm);
    \draw[fill] (-6,6) circle (.05cm);
    \draw[fill] (4,8) circle (.05cm);
    \draw[fill] (4,4) circle (.05cm);

    \draw (0,13) node {$\overline\alpha$};
    \draw (-7,6) node {$\overline\beta$};
    \draw (5,8) node {$(\overline a_i)_{\overline p_i\mapsto\overline t}$};
    \draw (5,4) node {$\overline a_i$};
    \draw (0,-1) node {$\overline b_i$};

    \draw (-1,-2) node {\bf Figure 4.1};

    \end{tikzpicture}
    
\end{center}

We still have that $\{\overline\beta,\overline\alpha,(\overline a_i)_{\overline p_i\mapsto\overline t}\}$ is a subset of a distributive lattice order-isomorphic to $D_{i+1}$.  Since $\overline\beta\lessdot\overline\alpha$, by semimodularity $(\overline a_i)_{\overline p_i\mapsto\overline u}\lessdot(\overline a_i)_{\overline p_i\mapsto\overline t}$, and hence $\overline u\lessdot\overline t$. \qed

Now $\overline d_i\vee(\overline a_i)_{\overline p_i\mapsto\overline u}$ exists (by Lemma 4.4 for $D_{i+1}$ and Claim 2) and equals $(\overline a_i)_{\overline p_i\mapsto\overline t}$ since $(\overline a_i)_{\overline p_i\mapsto\overline u}\ngeq\overline d_i$; but $\overline d_i,(\overline a_i)_{\overline p_i\mapsto\overline u}\le(\overline a_i)_{\overline p_i\mapsto\widetilde{\overline t}}\in{\overline P}^{\overline P}$, so $\overline t\le\widetilde{\overline t}$, a contradiction.

\quad {\it Case 4. $1_{\overline C_{\overline p_i}}=\overline\alpha(\overline p_i)\gtrdot\overline d_i(\overline p_i)\gtrdot\overline d_{i+1}(\overline p_i)=0_{\overline C_{\overline p_i}}$}

This implies $\h(\overline P)=2$.

Because $\overline P$ is trebled, there exists $\overline v\in\overline P-\{\overline\beta(\overline p_i),\overline d_i(\overline p_i)\}$ such that $0_{\overline C_{\overline p_i}}\lessdot\overline v\lessdot 1_{\overline C_{\overline p_i}}$.

Now $\overline\beta(\overline p_i)\lessdot\overline\alpha(\overline p_i)$; $\overline\beta(\overline p_i)\ne\overline d_i(\overline p_i)$; and $\overline d_i(\overline p_i)\in\overline C_{\overline p_i}$, so $\overline\beta(\overline p_i)\notin\overline C_{\overline p_i}$.  Since $\overline\alpha\in\overline D_i$, $\overline\beta\ge\overline d_{i+1}$, and $\overline\alpha\ge\overline a_i$, we have for all $\overline p\in\overline P$ (noting that $\overline\beta(\overline p)=\overline\alpha(\overline p)$ for all $\overline p\in P-\{\overline p_i\}$)
$$
\overline\beta(\overline p)=\begin{cases}
\overline\beta(\overline p_i)\ \text{if}\ \overline p=\overline p_i\text{;}\\
0_{\overline C_{\overline p_i}}\ \text{if}\ \overline p<\overline p_i\text{;}\\
1_{\overline C_{\overline p}}\ \text{if}\ \overline p\nleq \overline p_i\text{;}
\end{cases}
$$
\noindent and
$$
\overline\alpha(\overline p)=\begin{cases}
\overline\alpha(\overline p_i)\ \text{if}\ \overline p=\overline p_i\text{;}\\
0_{\overline C_{\overline p_i}}\ \text{if}\ \overline p<\overline p_i\text{;}\\
1_{\overline C_{\overline p}}\ \text{if}\ \overline p\nleq \overline p_i\text{.}
\end{cases}
$$

Then $\overline\alpha_{\overline p_i\mapsto\overline v}=\overline\beta_{\overline p_i\mapsto\overline v}\in{\overline P}^{\overline P}$ since (Observation 4.3) $1_{\overline C_{\overline p}}=\overline\alpha(\overline p_i)=1_{\overline C_{\overline p_i}}$ for $\overline p\in\overset{\circ}{\downarrow} \overline p_i$ and $0_{\overline C_{\overline p_i}}\le1_{\overline C_{\overline r}}$ for all $\overline q,\overline r\in\overline P$ such that $\overline q<\overline r,\overline p_i$; and $1_{\overline C_{\overline q}}=1_{\overline C_{\overline r}}$ for all $\overline q,\overline r\in\overline P$ such that $\overline q<\overline r$.

Also, $\overline\alpha_{\overline p_i\mapsto\overline v}\lessdot\overline\alpha$ but $\overline\alpha_{\overline p_i\mapsto\overline v}\ne\overline\beta$, and $\overline d_{i+1}\le\overline\alpha_{\overline p_i\mapsto\overline v}$ but $\overline d_i\nleq\overline\alpha_{\overline p_i\mapsto\overline v}$.  By uniqueness, $\overline\alpha_{\overline p_i\mapsto\overline v}=\overline\beta$, a contradiction.

We conclude that $\phi(a_i)=\overline a_i$, so $\phi(b_i)=\overline b_i$ and hence by Proposition 4.6 $\phi[D_{i+1}]=\overline D_{i+1}$. \qed

\begin{corollary} Assume $P$ and $\overline P$ are trebled ordered sets.  Assume that
$$
|\{C_p\mid p\in P\}|=1.
$$

Then $\phi[D_{nh}]=\overline D_{\overline n\overline h}$. 
\end{corollary}

\proof We have $D_0=\{d_0\}$, $\overline D_0=\{\overline d_0\}$, and $\phi(d_0)=\overline d_0$.  This starts the induction. \qed

{\it Proof of Theorem 4.1.}  Assume $P$ and $\overline P$ are finite trebled ordered sets such that $P^P\cong{\overline P}^{\overline P}$.  We first prove $P\cong\overline P$.

We are done if $P$ or $\overline P$ is an antichain, so assume $h,\overline h\ge1$.

By Corollary 4.9, there is an order-embedding of ${\bf (h+1)}^P$ into ${\bf (\overline h+1)}^{\overline P}$, so $|{\bf (h+1)}^P|\le|{\bf (\overline h+1)}^{\overline P}|$.  By symmetry, $|{\bf (\overline h+1)}^{\overline P}|\le|{\bf (h+1)}^P|$, so they are equal and the aforementioned order-embedding is an order-isomorphism.

Taking ordered sets of join-irreducibles, we get
$$
P^\partial\times{\bf h}\cong{\overline P}^\partial\times{\bf{\overline h}}.
$$
\noindent The height is $h+h-1=\overline h+\overline h-1$, so $h=\overline h$.  Thus we can cancel the non-empty factor ${\bf h}$ (\cite[(4.3)]{LovFG}, \cite[p. 89]{JonMcKHB}, or \cite[Theorem 3.2.5]{DufGH}) to get $P^\partial\cong\overline P^\partial$, or $P\cong\overline P$. 

Now use the corollary in Section 3. \qed

Constants do not necessarily go to constants: Consider $(\bf2+\bf2)^{\bf 2+ \bf2}$.  This is order-isomorphic to $\bf3\times\bf3+\bf3\times\bf3+\bf3\times\bf3+\bf3\times\bf3$, but only two of these components can be the ones with the constant maps.

The next step would be to eliminate the hypothesis of being trebled, or at least make it ``doubled".  The second author believes the only place where this hypothesis can't be reduced to ``doubled" in this proof is the first part of the proof of Proposition 4.8.

Another idea is to see if $P^P\cong Q^Q$ implies $\widetilde P^{\widetilde P}\cong \widetilde Q^{\widetilde Q}$ for finite ordered sets $P$ and $Q$.

We could have our induction step include the statement that $\phi(a_i)=\overline a_i$---which would mean $\phi(b_i)=\overline b_i$ and hence $\phi[D_{i+1}]=\overline D_{i+1}$---and hope that this enables us to show that $\phi(a_{i+1})=\overline a_{i+1}$.

{\sl Acknowledgements} The first author thanks Dr. Tim Campion for giving him permission to publish his proof in Section 4.

%----------------------------------------------------------------------


\begin{thebibliography}{99}

\bibitem{BirDH} Birkhoff, Garrett: {\it Lattice Theory} {\rm (second edition)}, American Mathematical Society, New York City (1948).

\bibitem{BirFG} Birkhoff, Garrett: {\it Lattice Theory} {\rm (third edition)}, American Mathematical Society, Providence, Rhode Island (1967).

\bibitem{CamBD} Campion, T. (2024, October 15). Answer beneath: If $P$ and $Q$ are finite, non-empty posets and the poset of order-preserving maps $P^P$ is isomorphic to $Q^Q$, must $P$ be isomorphic to $Q$? Retrieved June 16, 2026, from {\tt https://mathoverflow.net/questions/479345/if-p-and-q-are-finite-non-}
{\tt empty-posets-and-the-poset-of-order-preserving-ma}

\bibitem{CamBE} Campion, T. (2025, November 18). Personal communication.

\bibitem{DavPriJB} Davey, B. A.; Priestley, H. A.: {\it Introduction to Lattices and Order} {\rm (second edition)}, Cambridge University Press, Cambridge (2002).

\bibitem{DufGH} Duffus, Dwight Albert: ``Toward a Theory of Finite Ordered Sets,'' Ph.D. thesis, University of Calgary (1978).

\bibitem{DufHD} Duffus, Dwight: ``Powers of Ordered Sets,'' {\it Order} {\bf 1} (1984), 83--92.

\bibitem{DufWilGI} Duffus, Dwight; Wille, Rudolf: ``A Theorem on Ordered Sets of Order-Preserving Mappings,'' {\it Proceedings of the American Mathematical Society} {\bf 76} (1979), 14--16.

\bibitem{FarBC} Farley, Jonathan David: ``Does the Endomorphism Poset $P^P$ Determine Whether a Finite Poset $P$ Is Connected? An Issue Duffus Raised in 1978,'' {\it Mathematica Bohemica} 
{\bf 148} (2023), 435--446.

\bibitem{GraAA} Gr\"atzer, G. {\it Lattice Theory: Foundation}. Birkh\"auser, Basel, Switzerland, 2011.

\bibitem{JonMcKHB} J\'onsson, Bjarni; McKenzie, Ralph: ``Powers of Ordered Sets: Cancellation and Refinement Properties,'' {\it Mathematica Scandinavica} {51} (1982), 87--120.

\bibitem{LovFG} Lov\'asz, L.: ``Operations with Structures,'' {\it Acta Mathematica Academiae Scientiarum Hungaricae} {\bf 18} (1967), 321--328.

\bibitem{SchAF} Schr\"oder, Bernd S. W.: {\it Ordered Sets, An Introduction with Connections from Combinatorics to Topology} {\rm (second edition)}, Birkh\"auser Verlag (2016). 

\end{thebibliography}
\end{document}